\documentclass[12pt]{amsart}

\usepackage{dsfont}
\usepackage{bm}
\usepackage{comment}
\usepackage{bbm}

\usepackage{fourier}
\newtheorem{thm}[equation]{Theorem}
\newtheorem{prop}[equation]{Proposition}
\newtheorem{lemma}[equation]{Lemma}
\newtheorem{cor}[equation]{Corollary}

\newtheorem{mainthm}{Theorem}

\theoremstyle{definition}
\newtheorem{defi}[equation]{Definition}

\theoremstyle{remark}

\newtheorem{rmk}[equation]{\bf Remark}

\renewcommand{\theequation}{\thesection.\ifnum\value{subsection}=0 1\else \arabic{subsection}\fi.\arabic{equation}}

\usepackage{comment}

\DeclareMathAlphabet{\mathcal}{OMS}{zplm}{m}{n}

\newcommand{\pssymb}{{\mathcal{PS}^s}}
\newcommand{\dcsymb}{{\mathcal{DC}^s}}
\newcommand{\ddsymb}{{\mathcal{DD}^s}}
\newcommand{\vs}{{\mathcal{B}^s_{\infty, \infty}}}
\newcommand{\bs}{{\mathcal{B}^{-s}_{1,1}}}

\usepackage[usenames]{color}
\usepackage[table,xcdraw]{xcolor}
\usepackage{tikz}
\usepackage[framemethod=tikz]{mdframed}
\definecolor{aliceblue}{rgb}{0.92, 0.93, 1.0}

\usepackage{mathtools}
\usepackage{enumitem}

\usepackage[colorlinks]{hyperref}

\DeclareMathOperator*{\esssup}{ess\,sup}

\newcounter{change}

\usepackage[colorinlistoftodos, textwidth=4cm]{todonotes}

\begin{document}


\title{Particle systems, Dipoles and Besov spaces of distributions}



\ \\

\author{Mateus Marra}
\email{mateus-marra@usp.br}
\author{Pedro Morelli}
\email{pedromorelli@usp.br}
\author{Daniel Smania}
\email{smania@icmc.usp.br}
\urladdr{https://sites.icmc.usp.br/smania/}

\address{Departamento de Matemática, Instituto de Ciências Matemáticas e de Computação (ICMC), Universidade de São Paulo (USP), Avenida Trabalhador São-carlense, 400,  
São Carlos, SP, CEP 13566-590, Brazil}



\begin{abstract}  We define distributions on an abstract measure space endowed with a 
sequence of partitions, and introduce analogues of Besov spaces with 
negative smoothness in this setting. 
In particular, we describe these spaces of distributions using 
unconditional Schauder bases consisting either of Haar wavelets or of 
pairs of Dirac masses (dipoles). 
This framework allows us to obtain duality results between Besov spaces 
of negative smoothness and H\"older spaces of functions with respect to 
an appropriately defined pseudo-metric.
\end{abstract}



 \maketitle

\setcounter{tocdepth}{1}
\tableofcontents



\section{Introduction}

Since their introduction \cite{besov_original}, Besov spaces have become a fundamental object in functional analysis due to their ability to precisely characterize the regularity of functions,  making them useful in various applications.

In particular, Besov spaces with negative smoothness are spaces of 
\emph{distributions}, and they play an important role. 
The theory of Besov spaces of distributions is very well developed 
in $\mathbb{R}^n$ and on manifolds. 
However, in an abstract setting such as a general measure space, 
the classical theory of distributions is not available, 
since the usual spaces of test functions, -spaces of 
$C^\infty$ functions, are not even defined there.

\emph{Our goal is to develop and study Besov spaces with negative smoothness 
in a highly irregular setting: measure spaces endowed with a certain sequence 
of partitions (a good grid).}

  
An unconditional Schauder basis consisting of unbalanced Haar wavelets, 
as defined by Girardi and Sweldens \cite{Sweldens1997}, 
will be an essential tool here. The idea of decomposing spaces of functions or distributions using simple 
building blocks is, of course, central in harmonic analysis: Fourier series, 
Schauder bases, wavelets, and atomic decompositions are well-known 
manifestations of this principle. 
See, for instance, Wilson's decomposition of Hardy spaces \cite{wilson_hardy}, 
the decomposition of classical Besov spaces using Souza's atoms by de Souza 
\cite{souza-atomicdec, souza2}, and the influential work of 
Frazier and Jawerth \cite{frazier}.  

However, developing harmonic analysis in settings such as fractals 
(Yang \cite{yangdachun_fractals}), homogeneous spaces (Han,  Lu, and  Yang   \cite{hanluyang} and Han,  Han,  He, Li and Pereyra \cite{per2}), quasi-metric spaces (Kairema, Li, Pereyra, Ward \cite{per}), 
or  measure spaces endowed with good grids S. \cite{smania_2022-PDE}, 
where classical analytic tools are not available, has proved to be difficult. 
In these environments, it seems that atomic decompositions and Haar wavelets 
are particularly suitable tools.

A significant result we obtain, which to the best of our knowledge is new, 
is that Besov spaces with negative smoothness admit an unconditional 
Schauder basis consisting of \emph{dipoles}, that is, distributions of the 
form $\delta_x - \delta_y$, where $\delta_x$ denotes the Dirac mass at $x$.  

Our results also yield duality statements for Besov spaces that parallel 
those in the classical setting of $\mathbb{R}^n$. 
In particular, we relate the dual of H\"older continuous functions 
(with respect to a certain pseudo-metric) to Besov spaces with negative 
smoothness.  

It is worth mentioning that, in our setting, the proofs are remarkably 
elementary.

The minimal structural requirements we impose on the underlying measure 
space make the resulting theory particularly useful when the space is 
highly irregular. 
For instance, we are primarily interested in using Besov spaces of 
distributions to study the ergodic theory of dynamical systems through 
the action of transfer operators on these spaces. 
M. and S.\ \cite{aniso} used this framework to analyze the action of 
transfer operators on anisotropic spaces of distributions defined on 
symbolic spaces (such as $\{0,1\}^{\mathbb{Z}}$ endowed with a 
Bernoulli measure).


\section*{Acknowledgements} 
M.M. was supported by CAPES-Brazil. P.M. was supported by FAPESP-Brazil   2022/05300-1. D.S.  was financed  by the S\~ao Paulo Research Foundation (FAPESP), Brasil, Process Number 2017/06463-3,  and Bolsa de Produtividade em Pesquisa CNPq-Brazil 311916/2023-6.

\section{Main results and plan of the paper}
\noindent

The main results of this work fit into a broader trend of developing 
harmonic analysis on phase spaces with very low regularity. 
The minimal structure we require is a measure space endowed with a 
\emph{good grid}, that is, a sequence of partitions satisfying mild 
assumptions (see Section~3 and also S. \cite{smania_2022-PDE}).

We define a scale of Besov spaces of \textbf{distributions} 
(that is, with negative smoothness) $B^{-s}_{1,1}$ and Banach spaces of functions $\vs$, with $s> 0$.  To this end, we view distributions as formal sums of unbalanced Haar 
wavelets and define the norm in terms of the coefficients of this 
representation. 
Of course, this is closely related to the classical methods of dyadic 
harmonic analysis (see Pereyra \cite{dya2, pereyra} and also 
L\'opez-S\'anchez, Martell, and Parcet \cite{dya}). The relation between  $B^{-s}_{1,1}$ and   $\vs$ is similar to the classical setting.

\begin{mainthm}\label{thmA}$\vs$ and $\bs$ are Banach spaces such that
\begin{enumerate}[label=\Roman*.]
    \item \label{thma1} There is a pseudo-metric  $d$ on $I$  such that the  space $\vs$, with $0< s< 1$, is the space of $s$-H\"older continuous functions $\mathcal{C}^s$  with respect to $d$.

    \item \label{thma2} The dual of  $\bs$  is $\vs$. Furthermore, there is a separable subspace of $\vs$ whose dual equals $\mathcal{B}^{-s}_{1,1}$.
\end{enumerate} 
\end{mainthm}

  This is done in  and  Section \ref{besov_infty}.  Next  we show that  Besov spaces with negative smoothness  have atomic  decomposition using  {\it Dirac masses and dipoles.}  This is done in   sections \ref{sec_ps},  \ref{sec:dipolebasis} and \ref{sec:diracdipole}. Those Dirac masses $\delta_x$ are analogous to Dirac measures in the classical setting and  \textit{dipoles} are distributions  of the form
\[
\delta_{x}-\delta_{y}.
\]
We can also consider $\pssymb$ as the closed subspace of distributions in $(\mathcal{C}^s)^\star$  that are limits of linear combinations of Dirac masses. Those are called {\it particle systems}.  Perhaps surprisingly 


\begin{mainthm} \label{thm5}  Let $d$ be the pseudo-metric  as in  Theorem \ref{thmA}. We have 
\begin{enumerate}[label=\Roman*.]
    \item \label{thmb1} $\mathcal{B}^{-s}_{1,1}=\pssymb$ with equivalence of norms. 
    
    \item \label{thmb2} There is a unconditional Schauder basis of  $\mathcal{B}^{-s}_{1,1}$ that contains only dipoles and a single Dirac mass.

    \item \label{thmb3} Given $\psi \in \mathcal{B}^{-s}_{1,1}$ define
$$\|\psi\|_{\ddsymb}=\inf \sum_{i=0}^\infty |c_i|  +\sum_{j=0}^\infty |b_j|\hat{d}(x_j,y_j)^s,$$
where the infimum runs over all possible representations of $\psi$ of the form 
\begin{equation} \psi = \sum_{i=0}^\infty c_i\delta_{z_i} +\sum_{j=0}^\infty b_j (\delta_{y_j}-\delta_{x_j}).\end{equation}
Then $\|\cdot\|_{\ddsymb}$ and $\|\cdot\|_{\mathcal{B}^{-s}_{1,1}}$ are equivalent.
\end{enumerate}
\end{mainthm}

The {\it metric}  $\hat{d}$ is closely related with the pseudo-metric $d$, and  we postpone  its definition to Section \ref{sec_ps}.


\section{Preliminaries}
\label{preliminaries}

\subsection{Good Grids} A measure space with a \textit{good grid} is a  probability space $(I, \mathcal{A}, \mu)$ endowed with a nested sequence of finite $\mathcal{A}$-measurable partitions (up to zero measure sets)  $\mathcal{P} = (\mathcal{P}^k)_{k\in \mathbb{N}}$ such that
    \begin{enumerate}
        \item $\mathcal{P}^0 = \{I\}$.
        \item Given $Q\in \mathcal{P}^k$, $k>0$, then $Q\subset P$, for some $P\in \mathcal{P}^{k-1}$
        \item There exists $\lambda, \lambda_* \in (0,1)$ such that, if $Q\subset P$, with $Q\in \mathcal{P}^{k+1}$ and $P\in \mathcal{P}^{k}$, we have 
        \[
        \lambda_* \leq \dfrac{|Q|}{|P|} \leq \lambda.
        \]
        \item  The $\sigma$-algebra $\mathcal{A}$ is generated by $\cup_k \mathcal{P}^k$.
    \end{enumerate}
    The assumptions required above  are very mild, ensuring that a broad class of phase spaces falls within its framework (see, for instance the first section of \cite{Smania_2020}).

\begin{rmk}
    In property 2, $Q$ is known as a \textit{children} of $P$ and $\mathcal{P}$ the \textit{father} of $Q$. We denote by $\Omega_P$ the set of all children of $P$. If $\mathcal{P}$ is a good grid, then 
    \begin{equation}
    \label{def-cgr}
    C_{GR}:=\sup_{P \in \mathcal{P}} \# \Omega_P < \infty.
    \end{equation}
\end{rmk}

\subsection{Unbalanced Haar wavelets.}

Girardi and Sweldens \cite{Sweldens1997} constructed a unconditional basis of $L^p$, for $1< p< \infty$, associated to a sequence of partitions on a measure space, similar to the classical Haar basis associated to the sequence of dyadic partition of $[0,1]$.  We briefly recall their construction.   Let 
\[
\mathcal{H}_Q = \bigcup\limits_{j=0}^\infty \mathcal{H}_{Q,j},
\]
be a collection of pairs $(A,B)$, where $A\subset \Omega_Q$, $B\subset \Omega_Q$, and $\cup A$ and $\cup B$ are disjoint, such that $\mathcal{H}_{Q,j}$
are  defined recursively in the following way. Choose an total order  
$$P_1^Q \preceq P_1^Q   \preceq \cdots \preceq P_{n_Q}^Q$$
of $\Omega_Q$ and define 
$$\mathcal{H}_{Q,0}=(\{P_1^Q,\ldots, P^{Q}_{[n_Q/2]}\}, \{P^{Q}_{[n_Q/2]+1}, P^{Q}_{n_Q}\})$$
and for every pair
\[
(S_1,S_2) = (\{S_1^1, \ldots, S_{n_1}^1\},\{S_1^2, \ldots, S_{n_2}^2\}) \in \mathcal{H}_{Q,j},
\]
where 
$$S_1^1 \preceq \cdots \preceq S^1_{n_1} \preceq S^2_1 \preceq \cdots \preceq S_{n_2}^2,$$
we add the pairs
\[
\begin{array}{ccc}
     (\{S_1^1,\ldots,S^1_{[n_1/2]}\},\{S^1_{[n_1/2]+1},\ldots, S_{n_1}^1\})  \\ \\
     (\{S_1^2,\ldots,S^2_{[n_2/2]}\},\{S^2_{[n_2/2]+1},\ldots, S_{n_2}^2\})
\end{array}
\]
to $\mathcal{H}_{Q,{j+1}}$. 
Since the $\mathcal{P}$ is a good grid, there are only finitely many $j\in \mathbb{N}$ for which $\mathcal{H}_{Q,j}\neq \emptyset$.

\begin{defi}[Girardi and Sweldens \cite{Sweldens1997} ]  Define the {\it unbalanced Haar wavelet} associated to the pair  $(S_1,S_2)\in \mathcal{H}_Q$ as 
    \begin{equation}
    \label{haar-atom-supp-Q}
    \begin{array}{lll}
         \phi_{(S_1,S_2)} &=& \left(\dfrac{\sum\limits_{P\in S_1}{{1}}_P}{\sum\limits_{P\in S_1}|P|}-\dfrac{\sum\limits_{R\in S_2}{{1}}_R}{\sum\limits_{R\in S_2}|R|}\right).
    \end{array}
    \end{equation}
\end{defi}
Note that the set formed by all unbalanced Haar wavelets and the characteristic function of $I$ is an orthogonal set of function on $L^2(\mu)$.

\section{Besov spaces $\vs$ and  $\bs$  through Haar wavelets}
\label{besov_infty}

\subsection{Besov spaces of functions} Consider  a probability space $(I,\mathcal{A},\mu)$, where $I$ is the phase space, $\mathcal{A}$ is its $\sigma$-algebra and $\mu$ its probability. We will often denote $\mu(A)$  by $|A|$. Let $\mathcal{P}$ be a good grid on $I$.

\begin{defi}
\label{beesseinfinitoinfinito}
$\vs$, with $s>0$,  is the space  of  all functions 
$\psi \in L^\infty(I)$ that can be written as
\begin{equation} \label{seq}
\psi = c_I 1_I+\sum\limits_{k=0}^\infty\sum\limits_{Q\in \mathcal{P}^k}\sum\limits_{{(S_1,S_2)}\in \mathcal{H}_Q}|Q|^{1+s}c_{{(S_1,S_2)}}\phi_{{(S_1,S_2)}},
\end{equation} 
where $c_I, c_{(S_1,S_2)} \in \mathbb{C}$, and
\[
\|\psi\|_\vs := |c_I|+ \sup\limits_{k\geq 0}\sup\limits_{Q\in \mathcal{P}^k}\sup\limits_{{(S_1,S_2)}\in \mathcal{H}_Q} |c_{{(S_1,S_2)}}|<\infty.
\]
\end{defi}

It is easy to see that the partial sums  of the right hand side of  (\ref{seq}) converges (exponentially fast ) to $\psi$ in $L^\infty(I)$. It follows  that 
\begin{align}\label{cc}  c_{(S_1,S_2)} =\frac{1}{K_{(S_1,S_2)} |Q|^s } \int  \psi \phi_{{(S_1,S_2)}} \ d\mu,\end{align} 
where
$$K_{(S_1,S_2)} = |Q| \left(\dfrac{1}{\sum\limits_{P\in S_1}|P|}+\dfrac{1}{\sum\limits_{R\in S_2}|R|}\right)$$
for every $(S_1,S_2)\in \mathcal{H}_Q$. Note that  $2\leq K_{(S_1,S_2)}\leq 1/\lambda_*$. In particular the representation (\ref{seq}) is unique and

\begin{prop} We have that $\|\cdot\|_\vs$
is a complete norm on $\vs$. The space  $\vs$ is continuously embedded on $L^\infty(I)$.
\end{prop}

The definition of $\vs$ seems to be artificial, however we can see it as a  natural generalization of the Besov space $\vs[0,1]$. Indeed 

\begin{defi}\label{metric} Consider the 
pseudo-metric $d$ on $I$ defined by   $d(x,y)=|P|$ if there is $n$ and $P\in \mathcal{P}^n$ satisfying $x,y \in P\in \mathcal{P}^n$ but there are  $Q_1,Q_2\in \mathcal{P}^{n+1}$ with $x \in Q_1$, $y\in Q_2$ and $Q_1\neq Q_2$. Otherwise define $d(x,y)=0.$
\end{defi} 

\begin{prop}[Theorem A-\ref{thma1}]
A measurable function  $\psi$ is $s$-H\"older continuous almost everywhere on $I$, that is, there exists $C$ such that 
$$|\psi(x)-\psi(y)|\leq Cd(x,y)^s$$
for almost every $(x,y)\in I\times I$, if and only if $\psi \in \vs$. Furthermore the norm
$$\|\psi\|_{\mathcal{C}^s(I)} = \|\psi\|_{\infty} + \esssup_{(x,y)\in I\times I} \frac{|\psi(x)-\psi(y)|}{d(x,y)^s} 
$$
is equivalent to the norm $\|\cdot\|_\vs$.
\end{prop}

\begin{proof} Firstly we prove that  $\mathcal{C}^s\subseteq \vs$ and this inclusion is continuous. Indeed for each $Q\in \mathcal{P}$ choose $x_Q \in Q$ such that 
$$|\psi(x)-\psi(x_Q)|\leq \|\psi\|_{\mathcal{C}^s}d(x,x_Q)^s\leq \|\psi\|_{\mathcal{C}^s}|Q|^s$$
for almost every $y \in Q$. Consequently it follows from (\ref{cc}) that for $(S_1,S_2)\in \mathcal{H}_Q$
 $$ |c_{(S_1,S_2)}| \leq  \frac{C}{ |Q|^{s+1} } \int_Q  |\psi -\psi(x_Q)| \ d\mu \leq C\|\psi\|_{\mathcal{C}^s},$$
so $\psi \in \vs$ and $\|\psi\|_{\vs}\leq C\|\psi\|_{\mathcal{C}^s}$.   
It remains to show  that  $\vs \subset \mathcal{C}^s$ and this inclusion is continuous. 
Given $\psi\in \vs$, let $x,y \in I$. Then $J\in \mathcal{P}^{k_0}$, for some $k_0$, and such that $d(x,y)~=~|J|.$
By (\ref{seq}) we have
\begin{align*} 
&|\psi(y)- \psi(x)|\\
&\leq  \sum\limits_{k=k_0}^\infty\sum\limits_{
\substack{Q\in \mathcal{P}^k\\ Q\subset J}}\sum\limits_{{(S_1,S_2)}\in \mathcal{H}_Q}|Q|^{1+s}|c_{{(S_1,S_2)}}| |\phi_{(S_1,S_2)}(y)- \phi_{(S_1,S_2)}(x)|\\
&\leq C\|\psi\|_{\vs} \Big( |J|^s+  \sum\limits_{k> k_0} \big( \sum\limits_{
\substack{Q\in \mathcal{P}^k\\ y\in Q}}\sum\limits_{{(S_1,S_2)}\in \mathcal{H}_Q}|Q|^s+\sum\limits_{
\substack{Q\in \mathcal{P}^k\\ x\in Q}}\sum\limits_{{(S_1,S_2)}\in \mathcal{H}_Q}|Q|^s   \big)\Big)\\
&\leq C\|\psi\|_{\vs} \Big( |J|^s+  \sum\limits_{k> k_0} 2\lambda^{k-k_0}|J|^s\Big)\\
&\leq C\|\psi\|_{\vs} |J|^s\leq C\|\psi\|_{\vs} d(y,x)^s.
\end{align*} 
A similar argument gives $\|\psi\|_{L^\infty}\leq C \|\psi\|_{\vs}$,  so $\psi \in \mathcal{C}^s$ and $\|\psi\|_{\mathcal{C}^s}\leq C \|\psi\|_{\vs}$.
\end{proof}

\subsection{Test functions}
The simplest and most regular functions  in our setting are step functions that are linear combinations of characteristic functions of sets in the grid. 

\begin{defi}[Test functions] Denote by $V_{0,N}$  the linear subspace formed by all functions of the form 
    $$c_I 1_I+
    \sum\limits_{k=0}^N\sum\limits_{Q\in \mathcal{P}^k}\sum\limits_{{(S_1,S_2)}\in \mathcal{H}_Q}c_{{(S_1,S_2)}}\psi_{{(S_1,S_2)}}
    $$
     Let
    $V_0=\cup_N V_{0,N}.$ The space of {\it test functions} $$V_0=\cup_N V_{0,N}$$ is the set of functions $\psi$ with representation (\ref{seq}) such that $c_{(S_1,S_2)} = 0$ for all but {\it finitely}  many ${(S_1,S_2)}$.
\end{defi}

\noindent The linear space $V_0$  is a natural choice to the space of test functions for  distributions   in our setting.  Denote by $B^s_{\infty,\infty,o}$ the closure of $V_0$ on $\vs$.

\begin{prop}\label{oo} For every $\psi \in \vs$ with representation (\ref{seq}) the following statements are equivalent. 
\begin{itemize}
    \item[A.] $\psi \in B^s_{\infty,\infty,o}$
    \item[B.] We have $$\lim_k \  \sup_{Q\in \mathcal{P}^k} \ \sup_{(S_1,S_2)\in \mathcal{H}_Q}\  |c_{(S_1,S_2)}|   =0.$$
    \item[C.] The partial sums of the right hand side of (\ref{seq}) converges to $\psi$ on $\vs$.
    \item[D.] We have $$\lim_k \ \sup_{Q\in \mathcal{P}^k} \ \esssup_{(x,y)\in Q\times Q}  \frac{|\psi(x)-\psi(y)|}{d(x,y)^s}   =0.$$
    \end{itemize}
\end{prop}

\subsection{Besov spaces of distributions} A {\it distribution}  $\psi$ is a linear functional $$\psi\colon V_0\rightarrow  \mathbb{C}.$$

\noindent A formal series 
$$d_I1_I+\sum\limits_{k=0}^\infty \sum\limits_{Q \in \mathcal{P}^k} \sum\limits_{{(S_1,S_2)}\in \mathcal{H}_Q} d_{(S_1,S_2)}  |Q|^{-s}  \phi_{(S_1,S_2)}$$
defines a distribution if we integrate it against a test function (there is only a finite number of no vanishing terms). One can ask which of those define a continuous functional on $\vs$, that is, we have that  
$$d_I\int \psi \ d\mu+ \sum\limits_{k=0}^\infty \sum\limits_{Q \in \mathcal{P}^k} \sum\limits_{{(S_1,S_2)}\in \mathcal{H}_Q} d_{(S_1,S_2)}  |Q|^{-s}\int \psi \phi_{(S_1,S_2)} d\mu$$
indeed converges for all $\psi \in \vs$. Note that if (\ref{seq}) holds then
\begin{align} &d_I\int \psi \ d\mu+ \sum\limits_{k=0}^N \sum\limits_{Q \in \mathcal{P}^k} \sum\limits_{{(S_1,S_2)}\in \mathcal{H}_Q}  d_{(S_1,S_2)}  |Q|^{-s}\int \psi \phi_{(S_1,S_2)} d\mu \nonumber \\ 
&\label{rep} = d_Ic_I+ \sum\limits_{k=0}^N \sum\limits_{Q \in \mathcal{P}^k} \sum\limits_{{(S_1,S_2)}\in \mathcal{H}_Q} d_{(S_1,S_2)} c_{(S_1,S_2)} K_{(S_1,S_2)},
\end{align} 
holds for $N\in \mathbb{N}\cup \{\infty\}$.

\begin{defi} Given $s> 0$, we define $\bs$ as the space of all continuous linear funtionals $\varphi\in (\vs)^\star$  that can be written as 
\begin{equation}\label{rep2}
\varphi (\psi) = d_I \int \psi \ d\mu+\sum\limits_{k=0}^\infty \sum\limits_{Q \in \mathcal{P}^k} \sum\limits_{{(S_1,S_2)}\in \mathcal{H}_Q} d_{(S_1,S_2)}  |Q|^{-s}\int \psi \phi_{(S_1,S_2)} d\mu,
\end{equation}
with $$\|\varphi\|_\bs:= |d_I|+\sum\limits_{k=0}^\infty \sum\limits_{Q \in \mathcal{P}^k} \sum\limits_{{(S_1,S_2)}\in \mathcal{H}_Q}|d_{(S_1,S_2)}|<\infty.$$
\end{defi}

\noindent Due (\ref{rep}) we have

\begin{prop}\label{eq} The representation for  (\ref{rep2}) is unique for $\varphi \in \mathcal{B}^{-s}_{1,1}$, so we can represent it uniquely as a formal series  $$\varphi = d_I1_I+ \sum_{k=0}^\infty \sum_{Q \in \mathcal{P}^k} \sum_{{(S_1,S_2)}\in \mathcal{H}_Q} d_{(S_1,S_2)} |Q|^{-s} \phi_{(S_1,S_2)}.$$ 
Moreover  $\|\cdot\|_\bs$ is a complete norm on $\bs$,  the norms $\|\cdot\|_\bs$ and $\|\cdot\|_{(\vs)^\star}$ are equivalent norms  on $\mathcal{B}^{-s}_{1,1}$, and $\mathcal{B}^{-s}_{1,1}$ is a separable  closed subspace  of $(\vs)^\star$. Moreover the closure of $V_0$ in $(\vs)^\star$ is $\mathcal{B}^{-s}_{1,1}$. Indeed
$$\varphi = \lim_N \ d_I1_I+ \sum_{k=0}^N \sum_{Q \in \mathcal{P}^k} \sum_{{(S_1,S_2)}\in \mathcal{H}_Q} d_{(S_1,S_2)} |Q|^{-s} \phi_{(S_1,S_2)}$$
in  $\mathcal{B}^{-s}_{1,1}$
\end{prop} 
\begin{proof} Given $\varphi$ as above it follows from
(\ref{rep}) that 
$$d_{(S_1,S_2)}=|Q|^{1+s}K^{-1}_{(S_1,S_2)}\varphi(\phi_{(S_1,S_2)}),$$
so the representation is unique. It is easy to see that  $(\mathcal{B}^{-s}_{1,1},|\cdot|_\bs)$ is isometric to $\ell^1(\mathbb{N})$, so that space is complete and separable. By (\ref{rep}) we have that 
$$\|\varphi\|_{(\vs)^\star}= |d_I|+\sum\limits_{k=0}^\infty \sum\limits_{Q \in \mathcal{P}^k} \sum\limits_{{(S_1,S_2)}\in \mathcal{H}_Q}|d_{(S_1,S_2)}| K_{(S_1,S_2)},$$
so the norm of $\mathcal{B}^{-s}_{1,1}$ and $(\vs)^\star$ are equivalent on $\mathcal{B}^{-s}_{1,1}$ and consequently $\mathcal{B}^{-s}_{1,1}$ is closed in $(\vs)^\star$. Note that $V_0 \subset \mathcal{B}^{-s}_{1,1}$ and it follows from (\ref{rep}) that $V_0$ is dense in $\mathcal{B}^{-s}_{1,1}$.
\end{proof}

\begin{rmk} It is easy to see that the set of unbalanced Haar wavelets, in addition to $1_I$, is a unconditional Schauder basis of $\mathcal{B}^{-s}_{1,1}$. Haar wavelets are often Schauder basis  for Besov spaces in $\mathbb{R}^n$. See for instance Triebel \cite{triebel} and Oswald \cite{oswald} and  references therein. See also S. \cite{smania_2022-PDE} for results on Besov spaces with positive smoothness  for  measure spaces with good grids. 
\end{rmk}

\begin{prop}
    We have that $(\mathcal{B}^{-s}_{1,1})^\star=\vs$.
\end{prop}

\begin{proof}
Let $\rho$ be a continuous linear functional acting on  $\mathcal{B}^{-s}_{1,1}.$ Since
$$\|\phi_{(S_1,S_2)}\|_{\mathcal{B}^{-s}_{1,1}}=|Q|^s$$ we have
$$|\rho(\phi_{(S_1,S_2)})|\leq \|\rho\|_{(\mathcal{B}^{-s}_{1,1})^\star} |Q|^s.$$
Define 
$$\phi_\rho= \sum\limits_{k=0}^\infty \sum\limits_{Q \in \mathcal{P}^k} \sum\limits_{{(S_1,S_2)}\in \mathcal{H}_Q}  \frac{\rho(\phi_{(S_1,S_2)})}{K_{(S_1,S_2)}|Q|^s} |Q|^{s}  \phi_{(S_1,S_2)}.$$
Then $\|\phi_\rho\|_{\vs}\leq \|\rho\|_{(\mathcal{B}^{-s}_{1,1})^\star}$ and 
for  $\psi\in \mathcal{B}^{-s}_{1,1}$ we have
\begin{align*}
\rho(\psi)&=\lim_{k\rightarrow \infty}\rho \left(\sum_{n=0}^{k}\sum_{Q\in \mathcal{P}^n}\sum_{(S_1,S_2)\in \mathcal{H}_Q}|Q|^{-s} c_{(S_1,S_2)}\phi_{(S_1,S_2)}\right)\\
&=\lim_{k\rightarrow\infty}\sum_{n=0}^{k}\sum_{Q\in \mathcal{P}^n}\sum_{(S_1,S_2)\in \mathcal{H}_Q}|Q|^{-s} c_{(S_1,S_2)}\rho(\phi_{(S_1,S_2)})\\
&=\sum_{n\in \mathbb{N}}\sum_{Q\in \mathcal{P}^n}\sum_{(S_1,S_2)\in \mathcal{H}_Q}|Q|^{-s} c_{(S_1,S_2)}\rho(\phi_{(S_1,S_2)})\\
&=\psi(\phi).
\end{align*}
Since 
$$\rho \in \mathcal{B}^{-s}_{1,1}\mapsto \phi_\rho \in \vs$$
is a bounded linear bijective map this concludes the proof.
\end{proof}

\begin{thm}[Theorem A-\ref{thma2}]
\label{v0fecho}
$(\mathcal{B}^s_{\infty,\infty,o})^\star=\mathcal{B}^{-s}_{1,1}$.
\end{thm}

\begin{proof} Of course  $\mathcal{B}^{-s}_{1,1}\subset (\vs)^\star\subset (\mathcal{B}^s_{\infty,\infty,o})^\star$ and all inclusions are continuous. It remains to show that the inclusion of $\mathcal{B}^{-s}_{1,1}$ in $ (\mathcal{B}^s_{\infty,\infty,o})^\star$  is onto. Indeed  let $\tau \in (\mathcal{B}^s_{\infty,\infty,o})^*$ and $\psi \in \mathcal{B}^s_{\infty,\infty,o}$ with a representation as in (\ref{seq}) then by Proposition \ref{oo}.C
\begin{align*}
\tau(\psi)&=\lim_N  \tau\left(\sum_{k=0}^N \sum_{Q\in \mathcal{P}^k}\sum\limits_{{(S_1,S_2)}\in \mathcal{H}_Q}|Q|^{s+1}c_{(S_1,S_2)}\phi_{(S_1,S_2)}\right)\\
&=\lim_N \sum_{k\in\mathbb{N}}\sum_{Q\in \mathcal{P}^k}\sum\limits_{{(S_1,S_2)}\in \mathcal{H}_Q}|Q|^{s+1}c_{(S_1,S_2)}\tau(\phi_{(S_1,S_2)})\\
&= \sum_{k\in\mathbb{N}}\sum_{Q\in \mathcal{P}^k} \sum\limits_{{(S_1,S_2)}\in \mathcal{H}_Q}|Q|^{s+1}c_{(S_1,S_2)} \tau(\phi_{(S_1,S_2)}),
\end{align*}
and given that $\tau$ is a bounded  functional on $B^s_{\infty,\infty,o}$ we conclude 
$$\sum_{k\in\mathbb{N}}\sum_{Q\in \mathcal{P}^k} \sum\limits_{{(S_1,S_2)}\in \mathcal{H}_Q}  |Q|^{s+1}|\tau(\phi_{(S_1,S_2)})|<\infty,$$
so 
$$\phi = \sum_{k\in\mathbb{N}}\sum_{Q\in \mathcal{P}^k} \sum\limits_{{(S_1,S_2)}\in \mathcal{H}_Q}  \frac{|Q|^{s+1}}{K_{(S_1,S_2)}} \tau(\phi_{(S_1,S_2)}) |Q|^{-s}\phi_{(S_1,S_2)}$$
belongs to $\mathcal{B}^{-s}_{1,1}$. Now (\ref{rep}) gives
$$\phi(\psi)=\tau(\psi).$$
\end{proof}


\section{Dirac Masses,  Particle Systems and Dipoles}
\label{sec_ps}

 \subsection{Dirac Masses.}
 In this section, we prove the elements of $\bs$ can be  described by a combination of distributions  similar to the usual Dirac masses. Since $\vs$-observables are defined \textit{almost everywhere}, the usual definition of Dirac distributions does not make sense, since the evaluation at a point is meaningless.

\begin{defi}[Dirac masses]
\label{dirac-defi} Let $\hat{I}$ be the set of all  possible sequences $x=(Q_0,Q_1,\ldots)$, where, $Q_0=I$ and  $Q_{j+1}\in \Omega_{Q_j}$. We will write $x\in Q_j$ for every $j$. Define the {\it Dirac mass} $\delta_x$ associated to $x\in \hat{I}$ as
\begin{equation}
    \label{diracs-formula}
    \delta_x = \lim\limits_{j \to \infty} \dfrac{{{1}}_{Q_j}}{|Q_j|},
\end{equation}
where the limit is with respect to the norm on $\bs$. 
Note that $\delta_x$ is well-defined due to the following lemma.
\end{defi}

\begin{rmk} We can define a {\it metric} $\hat{d}$ on $\hat{I}$  similar to the pseudo-metric $d$ on $I$ as defined in Definition \ref{metric}. With this metric $\hat{I}$ is a Cantor set. There is a natural projection map $\pi\colon I \rightarrow \hat{I}$ and if we define the measure $\hat{\mu}=\pi^\star \mu$ then  $(\hat{I},\hat{d},\hat{\mu})$   is a 1-Ahlfors-regular metric space. 
\end{rmk}

\begin{lemma}
\label{generalized-diracs-lemma}
    Let $x=(Q_j)_{j\geq 0} \in \hat{I}$ and define, for all $j \geq 0$, the following element  $\phi_j \in \bs$
$$ \phi_j :=\dfrac{{{1}}_{Q_j}}{|Q_j|}$$    
    Then $(\phi_j)_{j \geq 0}$ is a convergent sequence  in $\bs$.
\end{lemma}

\begin{proof}
     It is enough to show that  $(\phi_j)_{j \geq 0}$ is a Cauchy sequence. Indeed, as
\[
\phi_{j+1} = \phi_j + \sum\limits_{\substack{(P_1,P_{2}) \in \mathcal{H}_{Q_j} \\ Q_{j+1}\in P_1\cup P_{2}}} \dfrac{|Q_{j+1} \cap P_1|\cdot |P_1|+|Q_{j+1} \cap P_2|\cdot |P_2|}{|P_1 \cup P_2|} \phi_{(P_1,P_2)}.
\]
\noindent Hence, we have
\begin{equation}
\label{problems}
\begin{array}{lll}
     \|\phi_{j+1}-\phi_j\|_\bs &= \displaystyle\sum\limits_{\substack{(P_1,P_{2}) \in \mathcal{H}_{Q_j} \\ Q_{j+1}\in P_1\cup P_{2}}} |Q_j|^s \cdot \dfrac{|Q_{j+1} \cap P_1|\cdot |P_1|+|Q_{j+1} \cap P_2|\cdot |P_2|}{|P_1 \cup P_2|} \\ \\ &\leq C \lambda^j,
\end{array}
\end{equation}
because $Q_j \in \mathcal{P}^j$. As the right hand side of \eqref{problems} is summable in $j$, we get that $(\phi_j)_{j\geq 0}$ is Cauchy sequence in $\bs$, as desired.
\end{proof}

 From the proof of the previous lemma, we have
\begin{equation}
    \label{dirac-nonbinary-expression}
    \delta_x = \sum\limits_{k=0}^{\infty} \sum\limits_{\substack{(P_1,P_2) \in \mathcal{H}_{Q_k} \\ Q_{k+1}\in P_1 \cup P_2}} \dfrac{|Q_{k+1} \cap P_1|\cdot |P_1|+|Q_{k+1} \cap P_2|\cdot |P_2|}{|P_1 \cup P_2|} \phi_{(P_1,P_2)}.
\end{equation}

Next we define the class of distributions that will be the building blocks of the atomic decomposition of $\bs$.

\begin{defi}
\label{finite-config-defi}
    A linear functional $\gamma : \vs \to \mathbb{C}$ is called a \textit{finite configuration of particles (FC)} if there exists $m_i \in \mathbb{C}$, $x_i \in \hat{I}$, for $i=1, \ldots , n$ such that
    \[
    \gamma = \sum\limits_{i=1}^n m_i \delta_{x_i}.
    \]
\end{defi}

From now on, all the series will converge on  the norm  of $(\vs)^\star$.

\begin{defi}
\label{wonder}
     A linear functional $\gamma : \vs \to \mathbb{C}$ is called a \textit{$s$-particle system $\pssymb$} if
    \begin{equation}
    \label{dell}
    \gamma = \lim_i  \gamma_i,
    \end{equation}
on  the norm  of $(\vs)^\star$. Here $(\gamma_i)$ is a sequence of finite configurations of particles. Of course  $\gamma \in (\vs)^\star$. An useful example of particle system that we are considering are the following:

\end{defi} 

\begin{defi}
    A linear functional $\gamma: \vs \to \mathbb{C}$ is called a \textit{dipole} if
    \[
    \gamma = \delta_x - \delta_y,
    \]
    for some $x,y \in \hat{I}$ with $x\neq y$.
\end{defi}

\section{Unconditional basis of dipoles} 
\label{sec:dipolebasis}
\begin{defi}
\label{dipolebasis} A  {\it 
 dipole basis} is an indexed family of dipoles (and a single additional Dirac mass) defined in the following way.  Firstly  for each $$P\in \mathcal{F}=\{I\}\cup \bigcup_{P\in \mathcal{P}} \bigcup_{(P_1,P_2)\in \mathcal{H}_P} \{P_1,P_2   \}$$ we choose a Dirac mass $x_P \in \hat{I}$ such that 
\begin{itemize}
    \item[A.] $x_P\in P$,
    \item[B.] If $P,Q \in \mathcal{F}$, $P \subset Q$ and $x_Q\in P$ then $x_P=x_Q$.
\end{itemize}
A dipole basis associated with this choice is the indexed family 
$$\{\delta_I\}\cup \displaystyle\bigcup_{Q\in \mathcal{P}}\bigcup_{(P_1,P_2)\in \mathcal{H}_Q}\{ \delta_{x_{P_1}}-\delta_{x_{P_2}}\}$$ 
\end{defi}
There are many possible choices of a dipole basis. From now on we fix one of them.
\begin{defi}
\label{dyp-config-defi}
    A distribution  $\gamma : V_0 \to \mathbb{C}$ is called a \textit{dipole configuration (DC)} if there exists $m_0 \in \mathbb{C}$ and, for all $P \in \mathcal{P}$, and coefficients $m_0, (m_{(P_1,P_2)})_{\substack{(P_1,P_2)\in \mathcal{H}_P}}$  such that
    \begin{equation}
    \label{dc-defi}
    \gamma =  m_0\cdot \delta_{I}+\sum\limits_{k=0}^\infty \sum\limits_{P\in \mathcal{P}^k}\sum\limits_{\substack{ (P_1,P_2) \in \mathcal{H}_P}} m_{(P_1,P_2)} \cdot (\delta_{x_{P_1}}-\delta_{x_{P_2}}).
    \end{equation}
 Note that $\gamma$ is well-defined since for $\varphi \in V_0$ only a finite number of dipoles in the dipole basis is no vanishing.
\end{defi}
 
\begin{rmk}
    
If a distribution  has a representation as a $\dcsymb$, then the coefficients $m_{(P_1,P_2)}$ are uniquely determined. Indeed, if
\[
\begin{array}{lllll}
     0 &=\gamma=&  m_0\cdot \delta_{I}+\displaystyle\sum\limits_{k=0}^\infty \sum\limits_{P\in \mathcal{P}^k}\sum\limits_{\substack{(P_1, P_2) \in \mathcal{H}_P}} m_{(P_1,P_2)} \cdot (\delta_{x_{P_1}}-\delta_{x_{P_2}}), 
\end{array}
\]
we see that $0 = \gamma ({{1}}_I) = m_0$ and by a recursive argument on $k$ one can show for all $k\geq 0$, $P \in \mathcal{P}^k$ and $(P_1,P_2) \in \mathcal{H}_P$,
\[
m_{(P_1,P_2)} = \gamma ({{1}}_{P_1}) = 0.
\]
\end{rmk}
We say that $\gamma$ belongs to $\dcsymb$ if
\[
\|\gamma\|_{\dcsymb}=|m_0|+\sum\limits_{k=0}^\infty \sum\limits_{P\in \mathcal{P}^k}|P|^{s}\sum\limits_{\substack{P_1, P_2 \in \mathcal{H}_P}} |m_{(P_1,P_2)}| < \infty.
\]

Our main result is the following

\begin{thm}[Theorem B-\ref{thmb1}]
\label{main result}
    We have that  $\mathcal{B}^{-s}_{1,1}=\pssymb=\dcsymb$.  Indeed there is $C > 0$, that does no depend on the chosen dipole basis such that 
    $$  \frac{1}{C}\|\varphi\|_{\dcsymb}\leq  \|\varphi\|_{\mathcal{B}^{-s}_{1,1}}\leq C \|\varphi\|_{\dcsymb}.$$
\end{thm}

\begin{lemma}\label{compnorms} There is $C> 0$, that does not depend on the chosen dipole basis, such that the following holds. Let $Q \in \mathcal{P}$ and $(Q_1,Q_2) \in \mathcal{H}_Q$. Then
    \[
    \|\delta_{x_{Q_1}}-\delta_{x_{Q_1}}\|_\bs \leq C \|\delta_{x_{Q_1}}-\delta_{x_{Q_1}}\|_\dcsymb.
    \]
\end{lemma}
\begin{proof}
    Write $x_{Q_i} = (Q^i_k)_k$ and $N=k_0(Q)$, where
    \[
    k_0(Q)=\min\{\ell:\exists P\in \mathcal{P}^\ell: Q\subset P\}
    \]
    By assumption, the Diracs agree up to the level $N$ so, by \eqref{dirac-nonbinary-expression},
   \begin{equation}
       \begin{array}{cc}
            \delta_{x_{Q_1}} - \delta_{x_{Q_2}} &= \displaystyle\sum\limits_{k=N}^{\infty} \sum\limits_{\substack{(P_1,P_2) \in \mathcal{H}_{Q_k^1} \\ Q_{k+1}^1\in P_1 \cup P_2}} \dfrac{|Q_{k+1}^1 \cap P_1|\cdot |P_1|+|Q_{k+1}^1 \cap P_2|\cdot |P_2|}{|P_1 \cup P_2|} \phi_{(P_1,P_2)} \\ 
        &+ \displaystyle\sum\limits_{k=N}^{\infty} \sum\limits_{\substack{(P_1,P_2) \in \mathcal{H}_{Q_k^2} \\ Q_{k+1}^2\in P_1 \cup P_2}} \dfrac{|Q_{k+1}^2 \cap P_1|\cdot |P_1|+|Q_{k+1}^2 \cap P_2|\cdot |P_2|}{|P_1 \cup P_2|} \phi_{(P_1,P_2)}
       \end{array}
   \end{equation}
   so we can compute the norm
   \begin{equation}
   \label{termAdef}
       \begin{array}{cccc}    \|\delta_{x_{Q_1}} - \delta_{x_{Q_2}}\|_\bs &\leq \displaystyle\sum\limits_{k=N}^{\infty} |Q_k^1|^s \sum\limits_{\substack{(P_1,P_2) \in \mathcal{H}_{Q_k^1} \\ Q_{k+1}^1\in P_1 \cup P_2}} \dfrac{|Q_{k+1}^1 \cap P_1|\cdot |P_1|+|Q_{k+1}^1 \cap P_2|\cdot |P_2|}{|P_1 \cup P_2|} \\ 
&+ \displaystyle\sum\limits_{k=N}^{\infty} |Q_k^2|^s \sum\limits_{\substack{(P_1,P_2) \in \mathcal{H}_{Q_k^2} \\ Q_{k+1}^2\in P_1 \cup P_2}} \dfrac{|Q_{k+1}^2 \cap P_1|\cdot |P_1|+|Q_{k+1}^2 \cap P_2|\cdot |P_2|}{|P_1 \cup P_2|}.
       \end{array}
   \end{equation}

Now, for any $k\geq N$, since $\mathcal{P}$ is a good grid, we have
\[
|Q_k^i|=|Q_N^i|\cdot \prod\limits_{m=N+1}^k \dfrac{|Q_m^i|}{|Q_{m-1}^i|}\leq \lambda^{k-N} \cdot |Q_N^i| = \lambda^{k-N} \cdot |Q|
\]
Hence,
\[
\begin{array}{llll}
     \|\delta_{x_{Q_1}} - \delta_{x_{Q_2}}\|_\bs &\leq&  C_{GR} \left[ \displaystyle\sum\limits_{k=N}^\infty |Q_k^1|^s + \sum\limits_{k=N}^\infty |Q_k^2|^s \right] \\ \\
     &\leq & 2 C_{GR} \left[\displaystyle\sum\limits_{k=N}^\infty \lambda^{s(k-N)}\cdot |Q|^s \right] \\ \\
     &\leq & 2 C_{GR} \left[\lambda^{-Ns} \displaystyle\sum\limits_{k=N}^\infty \lambda^{ks} \cdot |Q|^s \right] \\ \\
     &=& 2 C_{GR} \left[\lambda^{-Ns} \dfrac{\lambda^{Ns}}{1-\lambda^s} \cdot |Q|^s \right] \\ \\
     &=&  \dfrac{2 C_{GR}}{1-\lambda^s} \cdot |Q|^s.
\end{array} 
\]

as desired.
\end{proof}

Consider the dyadic grid on $[0,1]$, $\mathcal{C}^s[0,1]$ the space of Hölder functions on $[0,1]$ and, for $p\in [0,1]$, $\delta_p: \mathcal{C}^s[0,1] \to \mathbb{R}$ the usual Dirac mass concentrated on $p$. Of course, for the classic Besov spaces, the identity $\mathcal{B}^s_{\infty, \infty}[0,1] = \mathcal{C}^s[0,1]$ holds (see, for instance, Sawano \cite{sawano2018theory}) and one can prove that the sequence 
\[
\sum\limits_{k=0}^{2^n}\left(\dfrac{1}{2}\right)^n \delta_{k/2^n}
\]
of Riemann sums converges to $1_{[0,1]}$ in $(\mathcal{C}^s[0,1])^*$. The next lemma is a generalization of the technique in the setting of abstract measure spaces with a good grid. In this case, the sequence of "Riemann sums" are defined recursively. 


\begin{lemma}\label{conta} For every 
$(P_1,P_2)\in \mathcal{H}_Q$, with $Q\in \mathcal{P}$, let 
$$m_{(P_1,P_2)}=\begin{cases} |P_1|, &if \ x_{P_2}=x_{P_1\cup P_2},\\
-|P_2|, &if \ x_{P_1}=x_{P_1\cup P_2}.\end{cases}$$
Suppose that $J \in \mathcal{F}$
\begin{itemize}
    \item[A.]  either belongs to  $\mathcal{P}^{k_0}$,
    \item[B.] or  there is $(P_1,P_2)\in \mathcal{H}_Q$, with $Q\in \mathcal{P}^{k_0}$ such that either $J=P_1$ or $J=P_2$. 
\end{itemize} 
Let $A^{0}_J=|J|\delta_{x_J}$ and for $k> k_0$
$$A^i_J= \sum_{\substack{P\in \mathcal{P}^{k_0+i}\\ P\subset J}}  |P|\delta_{x_P} $$
for every $i> 0$. Then
$$A^{i+1}_J  =A^i_J + \sum_{\substack{R\in \mathcal{P}^{k_0+i}\\ R\cap J\neq \emptyset}} \sum_{\substack{(P_1,P_2)\in \mathcal{H}_R\\P_1\cup P_2\subset J }}  m_{(P_1,P_2)}(\delta_{x_{P_1}}-\delta_{x_{P_2}}) $$
\end{lemma}

\begin{lemma}
     The sequence $A_J^i$ converges to $1_J$ in $(\vs)^\star$.
\end{lemma}

\begin{proof} Let
$M=\sup_{R\in \mathcal{P}}\#\mathcal{H}_R.$
By Lemma \ref{conta}
    \begin{align}
        \|A_J^{i+1}-A_J^i\|_{\dcsymb} &= \Big\| \sum_{\substack{R\in \mathcal{P}^{k_0+i}\\ R\cap J\neq \emptyset}} \sum_{\substack{(P_1,P_2)\in \mathcal{H}_R\\P_1\cup P_2\subset J }}  m_{(P_1,P_2)}(\delta_{x_{P_1}}-\delta_{x_{P_2}})\Big\|_{\dcsymb}\nonumber \\
        &\leq  M \sum_{\substack{R\in \mathcal{P}^{k_0+i}\\ R\cap J\neq \emptyset}} |R|^{1+s} \nonumber \\
        &\leq M \lambda^{si} |J|^s\sum_{\substack{R\in \mathcal{P}^{k_0+i}\\ R\cap J\neq \emptyset}} |R|\nonumber \\
       \label{dcs} &\leq C M \lambda^{si} |J|^{1+s}.
    \end{align}
    Lemma \ref{compnorms}  implies that $A^i_J$ is a Cauchy sequence in $\dcsymb$, $\mathcal{B}^{-s}_{1,1}$ and $(\vs)^\star$, so in particular it converges  to some distribution $\varphi \in \mathcal{B}^{-s}_{1,1}$.  We claim that $\varphi=1_J$ on $(\vs)^\star$. Indeed, note that 
  $$A_J^i (1_Q) \to \int_I 1_Q 1_J d\mu$$ for  every $Q \in \mathcal{P}$ by Theorem \ref{v0fecho}. Indeed
    \begin{enumerate}
        \item If $Q \cap J = \emptyset$, then $A_J^i (1_Q) = 0$, for all $i$ and, hence, \mbox{$A^i_J(1_Q) \to 0 =\int_I 1_Q 1_J d\mu$}.
        \item If $Q\subset J$, choose $n$ large enough so that $Q$ contains one of the sums from $A^n_J$. In this setting
    \[
    A^i_J (1_Q) = |Q|,
    \]
    for all $i \geq n$ and, hence $A_J^i (1_Q) \to 1_J$, as desired.
    \item If $J\subset Q$, then, for all $i$, $A^i_J(1_Q) = |J|$ and, hence, $A^i_J(1_Q) \to |Q\cap J|$.
    \end{enumerate}
It follows that 
$$\varphi(\phi_{(S_1,S_2)})=\int \phi_{(S_1,S_2)} 1_Q \ d\mu$$
 for every $(S_1,S_2)\in \mathcal{H}_R$ and  $R \in \mathcal{P}$.  But since $\varphi \in \mathcal{B}^{-s}_{1,1}$ it follows that if $\varphi$ has the representation (\ref{rep2}) then for every $\psi$ as in (\ref{seq}) we have that (\ref{rep}) implies

 \begin{align*} \varphi(\psi)&= c_I\int \psi \ d\mu+ \lim_N \sum\limits_{k=0}^N \sum\limits_{P \in \mathcal{P}^k} \sum\limits_{{(S_1,S_2)}\in \mathcal{H}_P} d_{(S_1,S_2)} c_{(S_1,S_2)} K_{(S_1,S_2)}\\
 &= c_I\int \psi \ d\mu+ \lim_N \varphi\big(\sum\limits_{k=0}^N \sum\limits_{P \in \mathcal{P}^k} \sum\limits_{{(S_1,S_2)}\in \mathcal{H}_P} c_{(S_1,S_2)} \phi_{(S_1,S_2)}\big)\\
  &= c_I\int \psi \ d\mu+  \lim_N \int 1_J  \Big(\sum\limits_{k=0}^N \sum\limits_{P \in \mathcal{P}^k} \sum\limits_{{(S_1,S_2)}\in \mathcal{H}_P} c_{(S_1,S_2)} \phi_{(S_1,S_2)} \Big) \ d\mu\\
  &=\int 1_J \psi \ d\mu.
 \end{align*} 
  
\end{proof}

\begin{cor} 
\label{bs C DC continouous} There is $C> 0$, that does not depend on the chosen dipole basis, such that the following holds.
     Let $Q \in \mathcal{P}$ and $(Q_1,Q_2) \in \mathcal{H}_Q$. Then
    \[
    \|\phi_{(Q_1,Q_2)}\|_{\dcsymb} \leq C\cdot \|\phi_{(Q_1,Q_2)}\|_\bs.
    \]
\end{cor}
\begin{proof} By previous lemmas 

    \begin{align*}
    \label{wavelets are dipoles}
\phi_{(Q_1,Q_2)}&=\dfrac{1_{Q_1}}{|{Q_1}|} - \dfrac{1_{Q_2}}{|{Q_2}|} \\
&= (\delta_{x_{Q_1}}-\delta_{x_{Q_2}})+ \dfrac{1}{|Q_1|}\sum_{i=0}^\infty (A^{i+1}_{Q_1}-A^{i}_{Q_1}) - \dfrac{1}{|Q_2|}\sum_{i=0}^\infty (A^{i+1}_{Q_2}-A^{i}_{Q_2}) 
    \end{align*}
and, hence, by (\ref{dcs})
\[
\|\phi_{(Q_1,Q_2)}\|_{\dcsymb} \leq C |Q|^s=C \|\phi_{(Q_1,Q_2)}\|_{\mathcal{B}^{-s}_{1,1}}
\]
\end{proof}


\begin{proof}[Proof of Theorem \ref{main result}]
It is enough to prove the following continuous inclusions
$$\pssymb \subset   \bs \subset \dcsymb \subset \pssymb. $$
 The inclusion $\pssymb \subset \bs$ holds since $\bs$ is a Banach space and the elements of $\pssymb$ are limits of sequences in $\bs$. Moreover the inclusion is continuous by the equivalence of the $\bs$ norm and the one from $(\vs)^\star$ (by Proposition \ref{eq}).

    The inclusion $\dcsymb \subset \pssymb$ holds because dipoles are a particular example of finite configuration of particles. Moreover, let $\gamma \in \dcsymb$. Then
    \[
    \gamma = m_0 \delta_I + \sum\limits_{k\geq 0} \sum \limits_{P \in \mathcal{P}^k} \sum \limits_{(P_1,P_2) \in \mathcal{H}_P} m_{(P_1,P_2)} (\delta_{x_{P_1}}-\delta_{x_{P_2}}).
    \]
    Therefore, 
    \begin{align*}
         \|\gamma\|_{\bs} &\leq |m_0| \cdot  \|\delta_{x_I}\|_\bs + \displaystyle \sum\limits_{k\geq 0} \sum \limits_{P \in \mathcal{P}^k} \sum \limits_{(P_1,P_2) \in \mathcal{H}_P} |m_{(P_1,P_2)}| \cdot \|\delta_{x_{P_1}}-\delta_{x_{P_2}}\|_\bs  \\ 
         &\leq |m_0| \cdot  \|\delta_{x_I}\|_\bs + \displaystyle \sum\limits_{k\geq 0} \sum \limits_{P \in \mathcal{P}^k} \sum \limits_{(P_1,P_2) \in \mathcal{H}_P} |m_{(P_1,P_2)}| \cdot C |Q|^s \\ 
         &\leq C_2 \|\gamma\|_{\dcsymb},
    \end{align*}
    
    \noindent where $C_2 = \max \{C, \|\delta_{x_I}\|_\bs \}$. Hence, the inclusion is continuous.

 The continous inclusion $\bs \subset \dcsymb$ holds because, by Corollary \ref{bs C DC continouous},  unbalanced Haar wavelets are dipole configurations with uniformly bounded $\dcsymb$-norm.  Moreover, if $\gamma \in \bs$,
    \[
    \gamma = \sum_{k=0}^\infty \sum_{Q \in \mathcal{P}^k} \sum_{{(S_1,S_2)}\in \mathcal{H}_Q} c_{(S_1,S_2)} \phi_{(S_1,S_2)}.
    \]
    Therefore,
   \begin{align*}
   \|\gamma\|_{\dcsymb} &\leq \displaystyle \sum_{k=0}^\infty   \sum_{Q \in \mathcal{P}^k} \sum_{{(S_1,S_2)}\in \mathcal{H}_Q} |c_{(S_1,S_2)}| \cdot \|\phi_{(S_1,S_2)}\|_{\dcsymb} \\ 
         & \leq \displaystyle \sum_{k=0}^\infty   \sum_{Q \in \mathcal{P}^k} \sum_{{(S_1,S_2)}\in \mathcal{H}_Q} |c_{(S_1,S_2)}| \cdot C |Q|^s \\
         &\leq C\|\gamma\|_\bs,
    \end{align*}
    so the inclusion is continuous. This completes the proof.
\end{proof}

\begin{cor}[Theorem B-\ref{thmb2}] Every dipole basis is a unconditional Schauder basis of $\mathcal{B}^{-s}_{1,1}$.
\end{cor}

\begin{cor}\label{dipolenorm} There is a constant $C> 1$ such that the following holds. If $x,y \in \hat{I}$, with $x,y \in P\in \mathcal{P}^n$ and $x \in Q_1\in \mathcal{P}^{n+1}$, $y \in Q_2\in \mathcal{P}^{n+1}$, with $Q_1\neq Q_2$ then
$$\frac{1}{C}  |P|^s \leq  \|\delta_x - \delta_y\|_{\mathcal{B}^{-s}_{1,1}} \leq C |P|^s.$$
In particular $\|\delta_x - \delta_y\|_{\mathcal{B}^{-s}_{1,1}}\sim \hat{d}(x,y)^s$.
\end{cor}
\begin{proof} It is easy to see that every dipole  like above belongs to a dipole basis. Since in this basis $\|\delta_x - \delta_y\|_{\dcsymb}=|Q|^s$, the Corollary follows from Theorem \ref{main result}.
\end{proof}

 \section{Atomic decomposition of $\mathcal{B}^{-s}_{1,1}$ through Diracs and Dipoles}
 \label{sec:diracdipole}

 Given $\varphi \in \mathcal{B}^{-s}_{1,1}$, we say that 

\begin{equation}\label{ddrep} \varphi = \sum_{i=0}^\infty c_i\delta_{z_i} +\sum_{j=0}^\infty b_j(\delta_{y_j}-\delta_{x_j})\end{equation}
 is an {\it Dirac-Dipole atomic representation} of $\varphi$ if both series converges in $\mathcal{B}^{-s}_{1,1}$ and 
 $$\sum_{i=0}^\infty |c_i|  +\sum_{j=0}^\infty |b_j|\hat{d}(x_j,y_j)^s < \infty.$$
This  is the {\it cost} of this representation. 
 Note that due Corollary \ref{dipolenorm} this implies that the r.h.s. of (\ref{ddrep}) indeed converges unconditionally in $\mathcal{B}^{-s}_{1,1}$. Due Theorem \ref{main result} every element of $\mathcal{B}^{-s}_{1,1}$ has a Dirac-Dipole representation.

 Define $\|\varphi\|_{\ddsymb}$ as the infimum over the costs of all possible  Dirac-Dipole atomic representations of $\phi$.

\begin{cor}[Theorem B-\ref{thmb3}: Dirac-Dipole Atomic decomposition]
The norms $\|\cdot\|_{\mathcal{B}^{-s}_{1,1}}
$ and $\|\cdot\|_{\ddsymb}$ are equivalent on $\mathcal{B}^{-s}_{1,1}$.
\end{cor}    
\begin{proof} By Corollary \ref{dipolenorm} there is $C$ such that 
$\|\varphi\|_{\mathcal{B}^{-s}_{1,1}}\leq C \|\varphi\|_{\ddsymb}$ for $\varphi \in \mathcal{B}^{-s}_{1,1}$. On the other hand we can chose a dipole basis and by Theorem \ref{main result} the opposite inequality holds.
\end{proof}

\bibliographystyle{plain}
\bibliography{references.bib}

\end{document}